\documentclass[12pt,letterpaper,titlepage]{amsart}
\usepackage{amsmath, amssymb, amsthm, amsfonts,amscd,xr}

\newcommand{\R}{\mathbb{R}}
\newcommand{\C}{\mathbb{C}}

\def\beq{\begin{equation}}
\def\eeq{\end{equation}}

\def\a{\alpha}

\def\d{\delta}
\def\e{\varepsilon}

\def\A {\mathcal A}

\def\H{\mathcal H}

\def\M{\mathcal M}

\newenvironment{res}
               {\begin{equation}
\begin{minipage}{0.85\textwidth}}
               { \end{minipage}\end{equation} }
\def\ber{\begin{res} }
\def\eer{\end{res}}

\numberwithin{equation}{section}
\newtheorem{thm}{Theorem}[section]
\newtheorem{theorem}[thm]{Theorem}

\newtheorem{lemma}[thm]{Lemma}
\newtheorem{lem}[thm]{Lemma}

\newtheorem{cor}[thm]{Corollary}

\newtheorem{proposition}[thm]{Proposition}

\newtheorem{dfn}[thm]{Definition}
\newtheorem{remark}[thm]{Remark}

\makeatletter
\def\section{\@startsection {section}{1}{\z@}{3.5ex plus 1ex minus
    .2ex}{2.3ex plus .2ex}{\large\bf}}
    \def\subsection{\@startsection{subsection}{2}{\z@}{3.25ex plus 1ex minus
 .2ex}{1.5ex plus .2ex}{\bf}}
\makeatother

\def\a{\alpha}

\def\e{\epsilon}

\def\ka{\kappa}
\def\om{\omega}
\def\d{\delta}

\def\Xf{\mathfrak{X}}
\def\Yf{\mathfrak{Y}}

\def\A{\mathcal{A}}

\def\Bc{\mathcal{B}}

\def\Ze{\mathcal{Z}}
\def\Cc{\mathcal{C}}

\def\M{\mathcal{M}}

\def\A{\mathcal{A}}

\def\H{\mathcal{H}}

\def\L{\mathcal{L}}

\def\P{\mathcal{P}}

\def\Ri{\mathcal{R}}

\begin{document}
\title[support of the Plancherel measure]
{The central support of the Plancherel measure of an affine Hecke algebra}
\dedicatory{Dedicated to Victor Ginzburg at the occasion of his
fiftieth birthday}
\author{Eric Opdam}
\begin{abstract}We give conceptual proofs of certain basic properties of the
arrangement of shifted root hyperplanes associated to a root system
$R_0$ and a $W_0=W(R_0)$-invariant real valued parameter function
on $R_0$. The method is based on the role of this shifted
root hyperplane arrangement for the harmonic analysis of affine Hecke
algebras. In addition this yields a conceptual proof of the description
of the central support of the Plancherel measure of an affine Hecke
algebra given in\ \cite{O}.
\end{abstract}
\address{Korteweg de Vries Institute for Mathematics\\
University of Amsterdam\\
Plantage Muidergracht 24\\
1018TV Amsterdam\\
The Netherlands\\
email: opdam@science.uva.nl}
\date{\today}
\thanks{During the preparation of this paper the author
was partially supported by a Pionier grant of
the Netherlands Organization for Scientific Research (NWO)}
\maketitle\tableofcontents
\today
\section{Introduction}
Let $V$ be a Euclidean vector space and let
$R_0\subset V^*$ denote a reduced, integral root system
with Weyl group $W_0$ (we do not assume that $R_0$ spans $V^*$).
Let $k:R_0\to\mathbb{R}$ be a $W_0$-invariant real function
on $R_0$. All that follows depends on this function, but we
usually suppress it from the notation.
\begin{thm}\label{thm:1}
Let $o=o(R_0,V,k;\cdot):V\to\mathbb{Z}$ be the function on
$V$ defined by
\begin{equation}
o(v):=|\{\a\in R_0\mid \a(v)=k_\a\}|-|\{\a\in R_0\mid \a(v)=0\}|-\operatorname{dim}(V)
\end{equation}
Then $o\leq 0$, and $o(v)=0$ for at most finitely many $v\in V$.
\end{thm}
Observe that Theorem \ref{thm:1} implies that the set of zeroes of $o$
is nonempty only if $R_0$ spans $V^*$.
\begin{thm}\label{thm:2}
Let $v\in V$ be such that $o(v)=0$. Then $-v\in W_0v$.
\end{thm}
The zeroes of $o$ in $V$ are called ``residual points'' (for
$R_0$ with respect to the parameter function $k$).
More generally let $L\subset V$ be an affine subspace. We define
\begin{equation}
o_L:=|\{\a\in R_0\mid \a|_L=k_\a\}|-|\{\a\in R_0\mid \a|_L=0\}|-\operatorname{codim}(L)
\end{equation}
Let us write $L=v_L+V^L$ where $V^L\subset V$ is a linear subspace
and where
$\{v_L\}= V_L\cap L$ with $V_L=(V^L)^\perp$.
Denote by $R_L\subset R_0$ the parabolic subsystem of
roots that are constant on $L$. Then it is clear that
$R_L\subset V_L^*$ and that $o_L=o(R_L,V_L,k|_{R_L};v_L)$.
Hence Theorem \ref{thm:1} implies that $o_L\leq 0$ and that $o_L=0$
iff $v_L\in V_L$ is an $R_L$-residual point. Hence $V_L^*$ is spanned by
$R_L$ if $o_L=0$. Thus Theorem \ref{thm:1} implies that
there exist only finitely many affine subspaces $L$ with $o_L=0$.

Affine subspaces $L\subset V$ with $o_L=0$ are called
``residual affine subspaces'', and the point $v_L\in L$ such that
$\{v_L\}= V_L\cap L$ is called the center of $L$.
The $W_0$-orbits of the (centers of the) affine residual subspaces
should be thought of as parameter deformations of weighted Dynkin
diagrams of nilpotent orbits. The residual affine subspaces or rather
their toric analogs called ``residual cosets'' play a mayor role
in the description of the support of the Plancherel measure of the affine
Hecke algebra, see Theorem \ref{thm:support}.
\begin{thm}\label{thm:3}
If $L\subset V$ is a residual affine subspace of positive dimension
with center $v_L$. Then $v_L$ is not a residual point for $R_0$.
\end{thm}
These theorems are not new but can be deduced by combining
results of \cite{HOH0} and the appendix A of \cite{O}.
This yields proofs based on an elaborate classification
of all residual subspaces for all irreducible root systems $R_0$
and parameter functions $k$, and then
a case-by-case verification of the required properties.
In the special case with
``equal parameters'' (i.e. $k$ is a constant function)
there are conceptual proofs available for Theorem \ref{thm:1} and
Theorem \ref{thm:3} (but not for Theorem \ref{thm:2} as far as I am
aware) based on the interpretation of residual cosets in relation
to nilpotent orbits of complex semisimple Lie algebras.
More precisely in this special case Theorem \ref{thm:1} is
a consequence of the Richardson Dense Orbit
Theorem \cite[Theorem 5.2.1]{C} (see Remark
\ref{rem:richardson}) whereas Theorem \ref{thm:3} follows from
the Bala-Carter Theorem \cite[Theorem 5.9.5]{C} (see
Remark \ref{rem:bala}).
In the general unequal parameter case the only proofs
that were known to me were based on the complete classification
of the residual affine subspaces (see \cite{HOH0}, \cite{O}).

The Plancherel formula for affine Hecke algebras
was derived in \cite{O} as an application of an elementary uniqueness result
for certain multivariable residue distributions applied to an integral
representation of the natural trace functional $\tau$ of the Hecke algebra.
The method gives complete information on the central support of the
Plancherel measure
(i.e. the projection of the support to the spectrum of the center)
of the affine Hecke algebra and important partial information on the
formal dimensions of the discrete series in terms of the central
character (a product formula). Yet the situation was unsatisfactory
since the Theorems \ref{thm:1}, \ref{thm:2} and \ref{thm:3}
(lacking general conceptual proofs) were necessary ingredients
in order to apply the residue lemma to the harmonic
analysis (see \cite{O}).

In the present paper we remedy the situation and show that
the residue method of \cite{O} can be made to work without
using the Theorems \ref{thm:1}, \ref{thm:2} and \ref{thm:3}
as input.
We will derive the central support of the Plancherel measure
of the Hecke algebra directly from the residue method
(see Theorem \ref{thm:support}), which is the main result
of this paper. In addition the Theorems
\ref{thm:1}, \ref{thm:2} and \ref{thm:3}
will follow from this as consequences of the positivity and finiteness
of the trace $\tau$ of the affine Hecke algebra
(see Section \ref{sec:root}).
We will thus
achieve conceptual proofs from basic principles for
all these results.
\begin{remark}
These theorems are known to hold for non-integral root systems as
well (see \cite{HOH0}). It seems likely that the methods of the
present paper also apply in this more general context
using degenerate affine Hecke algebras (as in \cite{HOH0}).
\end{remark}
It is a pleasure to thank Joseph Bernstein and
David Kazhdan for their support and help.
Their insightful questions and comments on the ``residue method''
of \cite{O} made me aware of the possibility of results
of this kind.
\section{Preliminaries on affine Hecke algebras}
This section is a quick review of affine Hecke algebras.
We refer the reader to \cite{EO} or \cite{O} for details
and unexplained notations.

Let $W=W_0\ltimes X$ be an extended affine Weyl group with translation
lattice $X$ and let $S\subset W$ be a set of simple affine reflections.
We work with the affine Hecke algebra $\H$ over $\C$ with basis
$(N_w)_{w\in W}$ where $N_s$ (with $s\in S$) satisfies the quadratic relations
\begin{equation}
(N_s-q(s)^{1/2})(N_s+q(s)^{-1/2})=0
\end{equation}
Here $q:S\to\R_+$ is positive real and constant on conjugacy
classes of simple reflections. We can uniquely extend $q$ to a
length multiplicative function $q:W\to\R_+$ such that
$q(\omega)=1$ for all elements $\omega\in W$ of length $0$.
We often write $N_e=1$.
\subsubsection{Commutative subalgebra $\A$}
We define $\A\subset\H$ to be the commutative subalgebra of $\H$
spanned by the Bernstein basis elements $\theta_x$ ($x\in X$). We
define $T=\operatorname{MaxSpec}(\A)$, a complex torus. We denote
by $T_{v}$ the vector group corresponding to $T$, and by $T_u$ the
compact form of $T$. We have the polar decomposition $T=T_vT_u$.
We assume given a $W_0$-invariant inner product on
$\operatorname{Lie}(T_v)$ which is rational on $Y$ (the lattice
dual to $X$). We thus consider $\operatorname{Lie}(T_v)$ as a
Euclidean space in which $W$ acts as an extended affine Weyl
group.
\subsubsection{Center $\Ze$}
The center $\Ze\subset\A$ of $\H$ is given by Bernstein's
description $\Ze=\A^{W_0}$.
\subsubsection{Star algebra}
The anti-linear anti-involution
$*$ of $\H$ given by $N_w^*=N_{w^{-1}}$ turns $\H$ into
a star-algebra.

Observe that $\Ze\subset\A\subset\H$, where $\Ze$ is a
star subalgebra but $\A$ is not (in general).
\subsubsection{positivity of trace}
We use the standard trace $\tau$
on $\H$ given by $\tau(N_w)=\delta_{w,e}$.
\begin{proposition}
The trace $\tau$ is positive definite
with respect to $*$, and $N_w$ is an orthonormal basis for the
inner product $(x,y):=\tau(x^*y)$.
\end{proposition}
\subsubsection{Eisenstein series}
Let $R_0\subset X$ be the finite root system attached to $\H$ as usual,
i.e. the affine root system of the Coxeter subgroup of $W$ generated by
the simple affine reflections is the affine extension of $R_0^\vee$.
We let $\Delta\in\C[T]$ denote the Weyl
denominator
\begin{equation}
\Delta(t)=\prod_{\a\in R_{0,+}}(1-\a(t)^{-1})
\end{equation}
We denote by $w_0$ the longest element of $W_0$.
\begin{proposition}\cite[Proposition 2.24]{EO}
There exists a unique holomorphic function $E:T\to\H^*$ ($\H^*$ is
the linear dual of $\H$) such that
\begin{enumerate}
\item $E_t(1)=\Delta(t)$.
\item For all $t\in T,\,h\in\H,\,x\in X$: $E_t(\theta_xh)=E_t(h\theta_x)=x(t)E_t(h)$
\end{enumerate}
\end{proposition}
\subsubsection{Integral formula for trace}\label{sec:trace}
We have the following theorem:
\begin{theorem}\cite[Theorem 3.7]{EO}
Let $p\in T_v$ be far in the negative chamber. For all $h\in\H$:
\begin{equation}
\tau(h)=\int_{t\in pT_u}E_t(h)\Delta(t)^{-1}\eta(t)
\end{equation}
where
\begin{equation}\label{eq:eta}
\eta(t)=q(w_0)^{-1}\frac{\Delta(t)\Delta(t^{-1})}{D(t)D(t^{-1})}dt
\end{equation}
with
\begin{equation}
D(t)=\prod_{\a\in R_{0,+}}(1-q_{\a^\vee}^{-1}\a(t)^{-1})
\end{equation}
\end{theorem}
\begin{remark}
The statement ``$p\in T_v$ far in the negative chamber''
means that $p$ is sufficiently far from all the walls so that
there are no pole hyperplanes of $\eta$ between $p$ and infinity
in the negative chamber.

This implies that we can make the distance of
$p$ to any wall arbitrarily large without crossing poles (thus
without changing the value of the integral).
\end{remark}
\begin{remark}
In the case $W=W(C_n^\vee)^a$, the affine extension of $C_n$, the
affine Hecke algebra actually allows $3$ parameters $q_i$. The
formulae for $\eta$ are slightly more complicated in this $3$
parameter case, but essentially the same. We ignore this case here
to keep notations simple, but we remark that all arguments apply
to this case equally well.
\end{remark}
\begin{cor}\label{cor:tauonA}
On the commutative subalgebra $\A=\C[T]$ the trace integral formula
simplifies to (with $a\in\A$, and $p\in T_v$
far in negative chamber):
\begin{equation}
\tau(a)=\int_{t\in pT_u}a(t)\eta(t)
\end{equation}
\end{cor}
\begin{remark}
The trace $\tau$ on $\A$ will be studied below. Important
ingredients that will play a role are the product structure of
$\eta$, the $W_0$-invariance of $\eta$, and the elementary
Proposition below.
\end{remark}
\begin{proposition}\label{prop:wp}
We view $T_v$ as a real vector space, and we consider the
hyperplane arrangement $\P$ in $T_v$ formed by the intersections
of the pole hyperplanes of $\eta$ with $T_v$. For each central
real subarrangement $\P_c\subset\P$ of this hyperplane arrangement
of poles in $T_v$ and for each chamber $C$ of $\P_c$, there exists
a $w\in W_0$ such that $wp\in C$.
\end{proposition}
\begin{proof}
This is obvious: $\P_c$ consists of a collection of shifted root
hyperplanes of the form $\a(t)=q_{\a^\vee}$ (with $\a\in R_0$) in
$T_v$ which have a point in common. Hence $C$ is a shift of a
union $C^\prime$ of Weyl chambers of $W_0$. Since $p$ is far in
the negative chamber we can choose its distance to any wall much
larger than the shifts. In this situation clearly $wp\in C$ iff
$wp\in C^\prime$. But $W_0$ acts transitively on its Weyl
chambers, so that there indeed exists a $w\in W_0$ such that
$wp\in C^\prime$.
\end{proof}
\section{Residue calculus}
Let us now more generally look at a linear functional
$F_{\omega,p}$ on $\A=\C[T]$ defined by
\begin{equation}
F_{\omega,p}(a)=\int_{pT_u}a(t)\omega(t)
\end{equation}
where $\omega(t)=P(t)/Q(t)dt$ with
$P(t),Q(t)$ products of the form
$P(t)=\prod_{m^\prime\in\M^\prime}(1-d_{m^\prime}^{-1}x_{m^\prime}(t))$
and $Q(t)=\prod_{m\in\M}(1-d_{m}^{-1}x_{m}(t))$. Here
$\M,\M^\prime$ are finite index sets on which we have defined functions
$\M\ni m\to (d_m,x_m)\in\C^\times\times X$  and $\M^\prime\ni
m^\prime\to (d_{m^\prime},x_{m^\prime}) \in\C^\times\times X$.

Here $p\in T_v$ is such that $pT_u$ does not meet any of the pole
hyperplanes of $\omega$. In other words, if we project each of the
pole hyperplanes of $\omega$ to $T_v$ along $T_u$ then $p$ is a
regular element with respect to this affine hyperplane arrangement
in $T_v$.

Let us define for $m\in\M$ the complex
codimension $1$ submanifold $L_m\subset T$ by
\begin{equation}
L_m=\{t\in T\mid x_m(t)=d_m\}
\end{equation}
and similarly we define $L_{m^\prime}$ for $m^\prime\in\M^\prime$.
We write
\begin{equation}
T^m=\{t\in T\mid x_m(t)=1\}
\end{equation}
then $L_m$ is a coset of the subgroup $T^m\subset T$. Similar
notations are used for $\M^\prime$.

If $L\subset T$ define is any subset then we define
\begin{equation}
\M_L:=\{m\in\M\mid L\subset L_m\}
\end{equation}
and similarly we define $\M^\prime_L$.
Using this we define the order
$i_L=i_{\omega,L}$ of $\omega$ along $L$
by
\begin{equation}
i_L:=|\M_L|-|\M_L^\prime|
\end{equation}
\begin{dfn}
An $\omega$-residual coset $L$ of $T$ is a connected
component of the intersection of a subcollection
of hypersurfaces $L_m$ for $m$ running in a subset of $\M$
(observe that such a component is a coset of a subtorus
of $T$) which in addition satisfies the property
\begin{equation}
o_L:=i_L-\operatorname{codim}(L)\geq 0
\end{equation}
The collection of $\omega$-residual cosets is denoted by
$\L=\L_\omega$.
\end{dfn}
\begin{dfn}\label{dfn:rescos}
Let $L\in\L$. We define its tempered form
$L^t$ as follows. We can write $L=bT^L$ for some
base point $b$ and subtorus $T^L\subset T$. Let $T_L\subset T$
denote the subtorus whose Lie algebra is orthogonal
to $\operatorname{Lie}(T^L)$ (here we use the Euclidean structure
on $\operatorname{Lie}(T_v)$, and the fact that $X$ is rational
with respect to this structure).
Then $L^t:=(L\cap T_L)T_u^L$, which is a compact real form of $L$.
The set $(L\cap T_L)$ is finite; we fix an element $r_L\in L\cap T_L$.
The projection of $L^t$ onto $T_v$ along $T_u$ is a point $c_L\in T_v$,
the \emph{center} of $L^t$ ($c_L$ is the vector part of $r_L$).
\end{dfn}
\begin{remark}
By definition $\L$ is finite. Let $\Cc$ be the finite set
the centers of the elements $L\in\L$.
\end{remark}
With these definitions we can now formulate the basic residue lemma:
\begin{lemma}\label{lem:dist}
There exists a unique collection
$\{\Xf_c\in C^{-\infty}(cT_u)\}_{c\in\Cc}$
of distributions such that
\begin{enumerate}
\item The support of $\Xf_c$is contained in
\begin{equation}
S_c=\cup_{L\in\L: c_L=c}L^t\subset cT_u
\end{equation}
\item For all $a\in\A$:
\begin{equation}
F_{\omega,p}(a)=\sum_{c\in\Cc}\Xf_c(a|_{cT_u})
\end{equation}
\end{enumerate}
\end{lemma}
\begin{remark}
Of course the collection $\{\Xf_c\}$ will depend on the chamber
$C$ which contains $p$ of the real affine hyperplane arrangement
$\P$ in $T_v$ formed
by the real projections $L_mT_u\cap T_v$ of the pole hyperplanes.
If we want to stress this we will write $\Xf_{p,c}$.
\end{remark}
\subsubsection{Local cycles}
We can be more precise about the nature of the distributions
$\Xf_c$ ($c\in\Cc$) by analyzing the existence proof of Lemma
\ref{lem:dist}. The result is a description of $\Xf_c$ as a
summation over the $L\in\L$ with $c_L=c$ and (for each $L$) a sum
of normal derivatives of boundary values on $L^t$ of certain
rational functions on $L$ from certain wedges with edge $L^t$.
This was done in \cite[Proposition 3.7]{O}. Let us describe the
result (as we will need it further on):

Let $L\in\L$. We define
\begin{equation}
\M^L=\{m\in \M-\M_L\mid L\cap L_m\not=\emptyset\}
\end{equation}
and
\begin{equation}
\M^{L,t}=\{m\in \M-\M_L\mid L^t\cap L_m\not=\emptyset\}\subset\M^L
\end{equation}
The set $\M^L$ describes the intersection of poles of
$\omega$ with $L$, and $M^{L,t}$ the poles of $\omega$ which
meet $L^t$.

For $\d>0$ and each $L$ which is a connected
component of an intersection of codimension $1$ cosets
$L_m\subset T$ with $m\in \M$, we denote by $\Bc_L(r_L,\d)$
a ball in $T_L$
with radius $\d$ and center $r_L$, and by
$\Bc^L_{v}(\d)$
a ball
with radius $\d$ and center $e$ in $T^L_{v}$.

Let $U^L(\d)\subset T^L$ be the open set $\{t\in T^L\mid\forall
m\in\M^L: t\overline{\Bc_L(r_L,\d)}\cap L_m=\emptyset\}$. Note
that $U^L(\d_1)\subset U^L(\d_2)$ if $\d_1>\d_2$, and that the
union of these open sets is equal to the complement of union of
the codimension $1$ subsets $r_L^{-1}(L\cap L_m)\subset T^L$ with
$m\in \M^L$.
\begin{proposition}\label{prop:cycle} Let $\e>0$ be such that for all
$m\in\M$ and $L\in\L$,
$L_m\cap\Bc_L(r_L,\e)\Bc^L_{v}(\e)T_u^L\not=\emptyset$ implies
that $L^{t}\cap L_m\not=\emptyset$.
There exist
\begin{enumerate}
\item[(i)] $\forall L\in\L$, a point
$\e^L\in\Bc^L_{v}(\e)\backslash\cup_{m\in\M^{L,t}}T^m$,
\item[(ii)] a $0<\d<\e$ such that $\forall L\in\L^\om,\ \e^LT^L_u
\subset U^L(\d)$, and
\item[(iii)] $\forall L\in\L$, a compact cycle
$\xi_L\subset\Bc_L(r_L,\d)\backslash\cup_{m\in\M_L}L_m$
of dimension
$\operatorname{dim}_{\mathbb{C}}(T_L)$,
\end{enumerate}
such that $\forall c\in\Cc^\om, \forall\phi\in
C^\infty(cT_u)$:
\begin{equation}
\Xf_c(\phi)=\sum_{\{L\mid c_L=c\}}\Xf_L(\phi),
\end{equation}
where
$\Xf_L$
is the distribution on $cT_u$ with support $L^t$
defined by $\forall a\in\A$:
\begin{equation}
\Xf_L(a)=\int_{\e^LT_u^L\times\xi_L}a\om.
\end{equation}
If $\M^{L,t}=\emptyset$ we may take $\e^L=e$.
\end{proposition}
\begin{remark}\label{rem:split}
We remark that the splitting of $\Xf_c$ as a sum
of $\Xf_L$ as given here has in general no intrinsic
meaning (it is not unique, and the cycles
$\xi_L$ are not uniquely determined). Of course the
ambiguity of $\Xf_L$ is restricted to distributions with
support in the intersection of $L^t$ with other
tempered residual cosets $M^t$.
\end{remark}
We remark that (i), (ii) and (iii) imply
that the functional $\Xf_L$ on $\A$ indeed defines a
distribution on $c_LT_u$, supported on $L^t$.
We look at it more closely in order to arrive at a
useful conclusion.

The cycle of integration of $\Xf_L$ is the product
$\e^LT_u^L\times\xi_L\subset T^L\times T_L$.
Given $L\in\L$ and $t=t^Lt_L\in T$ we write
\begin{equation}
\omega(t)=(P^L(t^Lt_L)/Q^L(t^Lt_L))(P_L(t_L)/Q_L(t_L))dt_L\wedge dt^L
\end{equation}
with $P_L$ and $Q_L$ the product over the factors
parametrized by $\M_L^\prime$ and $\M_L$ respectively,
and where $P^L$ and $Q^L$ are
the products of the remaining factors. Here we
have chosen the orientations for $T^L_u$ and $T_{L,u}$ such that
if $dt^L$ and $dt_L$ denote the holomorphic continuations
of the normalized volume forms of $T^L_u$ and $T_{L,u}$ then
$dt=dt_L\wedge dt^L$.

Hence we may write
\begin{equation}
\Xf_L(a)=\int_{\e^LT^L_u}I(a,t^L)dt^L
\end{equation}
where $I(a,t^L)$ denotes the inner integral
\begin{equation}\label{eq:inner}
I(a,t^L):=\int_{t^L\xi}a(t^Lt_L)P^L(t^Lt_L)/Q^L(t^Lt_L))(P_L(t_L)/Q_L(t_L))dt_L
\end{equation}
Hence
\begin{equation}
I(a,t^L)=D(aP^L/Q^L)|_{r_Lt^L}
\end{equation}
where $D\in\operatorname{Sym}(\operatorname{Lie}(T_L))$ is
a constant coefficient differential operator on $T_L$,
extended to $T$ by invariance, homogeneous of degree
\begin{equation}
o_L:=i_L-\operatorname{codim}(L)
\end{equation}
Hence $I(a,t)$
is a linear combination of (possibly higher order)
partial derivatives $D_\ka a$ of $a$ at $r_Lt$ in the direction of
$T_L$, with coefficients in the ring of
rational functions on $T^L$ which are regular outside the
codimension $1$ intersections $r_L^{-1}(L\cap L_m)$
\begin{equation}
I(a,t)=\sum_{\ka}f_\ka D_{\ka} a.
\end{equation}
Indeed, by the Leibniz rule
the coefficient $f_\kappa$ is the restriction to $L$
of a sum of normal derivatives $D^\nu_\kappa(P^L/Q^L)$ of degree
\begin{equation}
\operatorname{deg}(D^\nu_\kappa)=o_L-
\operatorname{deg}(\operatorname{D_\kappa})
\end{equation}
Hence $\Xf_L(a)$ is equal to the sum of the boundary value
distributions $\operatorname{BV}_{\e^L,f_\ka}$ of the meromorphic
coefficient functions, applied to the corresponding partial
derivative $D_\ka a$ of $a$, restricted to $L^{t}$:
\begin{equation}
\Xf_L(a)=\sum_\ka \operatorname{BV}_{\e^L,f_\ka}(D_\ka
a|_{L^{t}}).
\end{equation}
We see that $\Xf_L$ is a distribution supported in
$L^{t}\subset c_LT_u$, which only depends on $\xi_L$ and on the
component of $\Bc^L_{v}(\e)\backslash\cup_{m\in\M^{L,temp}}T^m$
in which $\e^L$ lies.

The following corollary is the main result of this subsection:
\begin{cor}\label{cor:form} Let $L\in\L$.
Suppose that $\Xf_L\not=0$ (recall that this depends on $p$,
and in general on the choices of the $\e^L$).
Recall that $T$ is finite quotient
of $T^L\times T_L$ by $K_L=T_L\cap T^L$, a finite abelian group.
There exists an element $b\in\C[T_L]^{K_L}\subset \C[T]$ such that
\begin{enumerate}
\item The vanishing order of $b$ at $r_L$ is $o_L$.
\item For all $a\in \C[T]$ we have $\Xf_L(ab)=BV_{\e^L,\omega^L}(a|_{L^t})$,
where we mean by
$BV_{\e^L,\omega^L}$ the boundary value distribution on $L^t$
defined by
\begin{equation}
BV_{\e^L,\omega^L}(f)=\lim_{\e^L\to 1}\int_{t\in\e^L T^L_u}f(r_Lt)\omega^L(r_Lt)
\end{equation}
\end{enumerate}
where for $t\in \e^LT^L_u$ we put
\begin{equation}\label{eq:omL}
\om^L(r_Lt):=P^L(r_Lt)/Q^L(r_Lt)d^Lt
\end{equation}
\end{cor}
\begin{proof} We have
\begin{equation}
I(ab,t)=D(abP^L/Q^L)|_{r_Lt}
\end{equation}
where $D\in\operatorname{Sym}(\operatorname{Lie}(T_L))$ is
a constant coefficient differential operator on $T_L$,
extended to $T$ by invariance, nonzero (by assumption)
and homogeneous of order $o_L$. It is then elementary that
we can choose $b$ so that
\begin{equation}
I(ab,t)=a(r_Lt)P^L(r_Lt)/Q^L(r_Lt)
\end{equation}
as required.
\end{proof}
\begin{cor}\label{cor:form2}
Suppose that $F_{\omega,p}$ is a (complex) measure on
$S=\cup_{c\in\Cc}S_c$ (in the sense of Lemma \ref{lem:dist}).
Suppose moreover that $\Xf_L\not=0$
and that $L^t$ is not contained in the tempered form $M^t$ of
a larger residual subspace $M$. Then $o_L=0$.
\end{cor}
\begin{proof}
Let $\phi$ be a smooth function on $cT_u$ (with $c$ being the
center of $L$) whose support is contained in the complement of the
union of the tempered residual cosets $M^t$ with $M\not=L$ and
with center $c$. Using the notion of approximating sequences (cf.
\cite[Lemma 3.5]{O}) we can construct for any $N\in\mathbb{N}$ a
sequence $\{a_n\}$ in $\A$ such that $D(a_n)|_{L^t}\to
D(\phi)|_{L^t}$ uniformly for all constant coefficient
differential operators $D$ on $T$ of order at most $N$, and
$D(a_n)|_{M^t}\to 0$  uniformly for all $M\not=L$ (we construct
$a_n$ as in the proof of \cite[Lemma 3.19(ii)]{O}). Suppose that
$o_L>0$ and let $b\in\mathbb{C}[T_L]$ be as in Corollary
\ref{cor:form}. Then the sequence $a_nb$ converges
uniformly to $0$ on the support of $F_{\omega,p}$ and by the
assumption that $F_{\omega,p}$ is a complex measure we conclude
that $F_{\omega,p}(a_nb)\to 0$. Moreover $\Xf_M(a_nb)\to 0$ (for
$M\not=L$) by the properties of the sequence $a_n$ and since all
$\Xf_M$ are finite order distributions (choose $N$ sufficiently
large). We conclude that $\Xf_L(a_nb)\to 0$. On the other hand,
Corollary \ref{cor:form} implies that
\begin{equation}
\Xf_L(a_nb)\to BV_{\e^L,\omega^L}(\phi)
\end{equation}
and from the assumption that $L^t$ is not properly contained
in another tempered residual coset $M^t$ it is clear that we may
choose $\phi$ in such a way that this expression is nonzero.
This is the required contradiction.
\end{proof}
\section{Non-cancellation results}
Recall the following results of \cite{O}:
\begin{lem}\label{lem:lines}(\cite[Lemma A.11]{O})
Assume that $\omega$ is such that $o_L=0$ for all
$L\in\L$ with $\operatorname{dim}(L)>0$.
Let $r\in\L$ be a residual point for $\omega$.
There exists a $L\in\L^\omega$ such that $\operatorname{dim}(L)=1$
and $r\in L$.
\end{lem}
\begin{thm}(\cite[Theorem 3.29]{O})\label{thm:lines}
Suppose that $\omega$
is such that
\begin{enumerate}
\item[1.] $o_L=0$ for all $L\in\L$.
\item[2.] If $L,M\in\L$ and $L\not=M$ then $L^t\not\subset M^t$.
\end{enumerate}
Let $r\in\L$ be a residual point. Then there exists a chamber $C$
of $T_v^{reg}:=T_v\backslash\cup_{m\in\M}T_v\cap{L_mT_u}$ such
that for $p\in C$ we have $\Xf_{p,r}(a)=d_{p,r}a(r)$ for some
$d_{p,r}\not=0$.
\end{thm}
\begin{proof} This is essentially just \cite[Theorem 3.29]{O},
although there the result was only formulated for the special case
of the integral formula of the trace $\tau$ of $\H$ restricted to
$\A$. The arguments apply to the general case as well. Let us
discuss the arguments briefly here.

When one computes $\Xf_{p,r}$ \emph{under the simplifying
assumption 2} one chooses
a path from $p$ to $e$ in $T_v$ which intersects each
pole hyperplane $L_m$ at most once and at points of the form
$r_{L_m}t^m$ (with $t\in T^m_v$ and regular for
$r_{L_m}^{-1}(L_m\cap L_l)$ with $l\in M^{L_m}$). When we move $p$
along this path to $e$ we see \emph{by virtue of assumption
1} that a residue at $r_{L,m}t$ is picked up which is of the form
$d_mF_{\omega^m,T^m}$ (notations as in equation (\ref{eq:omL})
(with $d_m$ a constant) applied
to the restriction of $a\in\A$ to $L_m=r_LT^m$
(and then transported to $T^m$ via the map $t\to r_{L_m}t$).
One now continues by moving
the regular points $t\in T^m_v$ to the identity $e$ of $T^m_v$ along
paths in $T^m$ as above etc. etc.
At the last step of this process one has one-dimensional integrals
on the one dimensional tori $T^L$ of the form $d_LF_{\omega^L,t^L}$
with $t^L\in T^L_v$ and $d_L$ a constant depending on $p$ applied to the
restriction of $a\in\A$ to $L=r_LT^L$.
If we want to know $\Xf_{p,r}$ we need to move the
$t^L\in T^L_v$ to $e\in T^L_v$ for all the residual lines containing
$r$ and see if a residue at $r$ is picked up. Then we add the result,
which will in general be a big summation of constants time the delta
distribution of $r$.

By the Lemma above it is clear that there exists at least one
residual line through $r$. But the problem is cancellation, as
we have to add several complex constants $d_L$ coming from the
residual lines containing $r$ in order to compute $\Xf_{p,r}$.
We need to show that these constants do not cancel at $r$ for at least
one choice of $p$.

The argument in \cite[Theorem 3.39]{O} is an induction with
respect to the dimension of $T$. Assume by induction that it is
true for tori of dimension $n-1$. Then it is shown in the proof of
\cite[Theorem 3.39]{O} that one can choose  $p$ and the paths from
$t^L$ to $e$ in the residual cosets $L$ of dimension at least $2$
in such a way that among the nonzero residues picked up at one of
the residual lines $L$ through $r$ (at a point $r_Lt$ with $t\in
T^L_v$) precisely one of these one-dimensional residual integrals
has the property that $r_v$ is between $(r_L)_v$ and $(r_L)_vt$.
\emph{Here we used that $(r_L)_v\not=r_v$, as a consequence of
assumption 2}.
Hence in the final step of the residue computations in which we
move the base points of the residual integrals in the one
dimensional residual cosets $L$ from $(r_L)t$ to $r_L$, we pick up
\emph{precisely one} nonzero residue at $r$ corresponding to a
first order pole of $\omega^L$.
This implies the result.
\end{proof}
The following result is a slight generalization which will be handy
for later applications.
\begin{lemma}\label{lem:nc}
Suppose that $\omega$ is such that
\begin{enumerate}
\item[1.] $o_L=0$ for all $L\in\L$
with $\operatorname{dim}{L}>0$.
\item[2.] If $L, M\in\L$ and
$L\not=M$ then
$L^t\not\subset M^t$.
\end{enumerate}
Let $r\in\L$ be a residual point with
$o(r):=o_{\{r\}}=k\geq 0$. Then there exists a chamber $C$ of
$T_v^{reg}:=T_v\backslash\cup_{m\in\M}T_v\cap{L_mT_u}$ such that
for $p\in C$ we have $\Xf_{p,r}(a)=Da(r)$ for some differential
operator $D$ of order $k$ with highest order term of the form
$d_{p,r}X^k(a)(r)$ where $X$ is a constant vector field on $T$
corresponding to a one parameter subgroup of $T$ in the direction
of one of the one dimensional residual cosets through $r$.
\end{lemma}
\begin{proof}
This is similar to the proof that was outlined above,
except that in the last stage of the computation of the
one-dimensional residues
we need to take the residue of $a\omega^L$ at $r$ where
$\omega^L$ has a pole of order $k$ at $r$. This proves the
result.
\end{proof}
\begin{remark}
Here we really need to assume that $o_L=0$ for all the higher
dimensional residual cosets. Without this assumption things
become more complicated and we do not have a simple
non-cancellation result.
\end{remark}
\section{Application of positivity}
We return to the situation of paragraph \ref{sec:trace}.
\begin{proposition}\label{prop:posmeas}
The restriction of $\tau$ to $\Ze$ defines a Radon probability measure
$\nu$ on $W_0\backslash T$ whose support (viewed as a $W_0$-invariant
subset of $T$) is contained in the union $S$ of the tempered
forms $L^t$ where $L$ runs over the $\eta$-residual cosets (note
that this is a $W_0$-invariant set since $\eta$ is $W_0$-invariant).
Moreover, if $t\in S$ is in the support of $\tau$ then
$t^*=\overline{t}^{-1}\in W_0t$.
\end{proposition}
\begin{proof}
In \cite{O}, paragraph 3.3.2, this was shown with $S$ replaced
by the larger collection $S^{\operatorname{qu}}$ of
$\eta/\Delta$-residual subspaces (the argument showing that $\nu$
is a probability measure is based on the positivity of $\tau$ on $\H$).
Observe that $S^{\operatorname{qu}}$ too is
a $W_0$-invariant collection of cosets (since $\Delta$ is
essentially skew invariant). In \cite{O}, subsection 3.4
it was remarked that the support of $\tau|_{\A}$ is actually contained
in $S$. We give a slightly different argument here.
By Lemma \ref{lem:dist} applied to the restriction
of $\tau$ to $\A$ (Corollary \ref{cor:tauonA}) we see that
$\tau$ defines distributions $\Xf_c\in C^{-\infty}(cT_u)$ with $c$
running over the centers of $\eta$-residual cosets, with $\Xf_c$
supported on $S_c\subset S^{\operatorname{qu}}_c$ (the extra
factor $\Delta$ in the denominator cancels on $\A$, this is
the point here). The $W_0$-average in each $W_0$-orbit of centers
$W_0c\in\Cc$ defines a $W_0$-invariant collection of $W_{0,c}$-invariant
distributions $\Yf_{c}\in C^{-\infty}(cT_u)$ supported on
$S_c\subset S^{\operatorname{qu}}_c$.
Hence the difference of the sum of this $W_0$-invariant collection
of distributions and $\nu$, which is the sum of $W_0$-invariant collection
of positive measures $\nu_c$ supported on $S^{\operatorname{qu}}_c$
(also $W_{0,c}$-invariant by definition) defines the $0$-functional on $\A$.
By the uniqueness assertion of the residue lemma \ref{lem:dist} we have
the equality $\nu_c=\Yf_c$ for each $\eta/\Delta$-residual center
$c\in\Cc^{\operatorname{qu}}$, implying the result.
\end{proof}
\begin{remark}
We stress that the proofs of these facts use no special
properties other than the $W_0$-invariance of $\eta$, the positivity
of $\tau$ on $\H$, and the
fact that the integral formula for $\tau$ on $\A$ simplifies
by cancellation of a factor $\Delta$. In particular we have not
used knowledge obtained by the classification of residual cosets.
\end{remark}
\begin{remark} As in the proof above we will
write $\nu_c$ for the $W_{0,c}$-invariant measure on $cT_u$
(with $c\in\Cc$) defined by restriction of $\nu$.
\end{remark}
Another important remark which follows from positivity is
the following. For any $h\in\H$ we define  a collection ``local
distributions'' $\Xf^h_c\in C^{-\infty}(cT_u)$
where $c$ runs over the set $\Cc^{\operatorname{qu}}$ of centers of
$\eta/\Delta$-residual cosets by applying the residue lemma \ref{lem:dist}
to the integral formula for the functional $\A\ni a\to\tau(ha)$ on $\A$
(as a special case we have $\Xf_c=\Xf_c^1$).
Again following \cite{O} we symmetrize these local distributions
over $W_0$ (as in the
proof of Proposition \ref{prop:posmeas}) in order to obtain a
$W_0$-invariant collection of distributions
$\Yf^h_c\in C^{-\infty}(cT_u)^{W_{0,c}}$ which is
uniquely characterized by
\begin{equation}
\tau(zh)=\sum_{c\in W_0\backslash\Cc^{\operatorname{qu}}}\Yf_c^h(z|_{cT_u})
\end{equation}
for all $z\in\Ze$.

By application of the positivity of the trace and the basic fact that
for any hermitian element $h$ of $\H$ the element
$h+\Vert h\Vert_o 1\in\H_+$ is a positive element, one can prove
the following fundamental result:
\begin{thm}\label{thm:abscont}\cite[Corollary 3.23]{O}
For all $h\in\H$ and all $c\in\Cc^{\operatorname{qu}}$
$\Yf_c^h$ is absolutely continuous with respect to $\nu_c$.
In fact we have for all continuous $W_{0,c}$-invariant
functions $\phi$ on $cT_u$ that
\begin{equation}
|\Yf_c^h(\phi)|\leq\Vert h\Vert_o\nu_c(\phi)
\end{equation}
where $\Vert h\Vert_o$ denotes the norm of $h$ in the
left regular representation (with respect to the Hilbert
norm on $\H$ defined by the $\tau$ and $*$ as in the first
section).
\end{thm}
\begin{cor}
For all $h\in\H$ the symmetrized local distributions $\Yf_c$
are supported on $S$ (as opposed to the larger set
$S^{\operatorname{qu}}$).
\end{cor}
\begin{dfn}(see \cite[Corollary 3.23]{O})
In particular, there exists a unique $\nu$-integrable $W_0$-invariant
function $t\to\chi_t\in\H^*$ on $T$, defined on the support of $\nu$,
which is for $\nu$-almost every $t\in S$ a positive tracial state on
$\H$ with central character
$W_0t$ (i.e. it is a positive trace, $\chi_t(1)=1$,
and $\chi_t(zh)=z(t)\chi_t(h)$ for all $t\in S$) such that
for all $h\in\H$:
\begin{equation}\label{eq:dec}
\tau(h)=\int_{S}\chi_t(h)d\nu(t)
\end{equation}
\end{dfn}
\begin{cor}\label{cor:projsup}
The support of $\nu$ on $W_0\backslash T$ is equal to the central
support of the Plancherel measure of $\H$.
\end{cor}
\begin{proof}
The local traces $\chi_t$ for $t\in\operatorname{Supp}(\nu) = S$
can be decomposed as a positive finite linear combination of
irreducible characters of $\H$, each having central character $W_0t$.
This further refinement of the decomposition (\ref{eq:dec}) is the Plancherel
formula of $\H$, and the (closure of) the set of irreducible characters $\H$
that arises in this way is the support of the Plancherel measure.
Hence the support of $\nu$ is the projection of the
support of the Plancherel measure to
$\operatorname{MaxSpec}(\Ze) = W_0\backslash T$.
\end{proof}
\begin{cor}\label{cor:Xmeas}
The local distributions $\Xf_c$ arising from the application
of Lemma \ref{lem:dist} to the restriction of $\tau$ to $\A$
are themselves positive Radon measures supported on $S$ such that
$\Xf=\sum_{c\in\Cc}\Xf_c$ is a (nonsymmetric)
probability measure on $T$.
\end{cor}
\begin{proof} Fix $c\in\Cc$.
We have  by definition for all
$a\in\A$ and $\phi\in C^{\infty}(cT_u)^{W_{0,c}}$:
\begin{equation}\label{eq:start}
\Yf^a_c(\phi)=|W_0/W_{0,c}|\int_{S_c}\phi(t)\chi_t(a)d\nu(t)
\end{equation}
Since $\chi_t(\cdot)$ is a positive tracial state for $\nu$-almost
all $t\in S_c$ which has central character $W_0t$, the GNS
construction implies that $\chi_t$ is $\nu$-almost everywhere a
positive real linear combination of irreducible (and thus finite
dimensional) unitary characters of $\H$ with central character
$W_0t$. It follows (using the finite dimensionality of the
irreducibles) that there exists an essentially unique (i.e. modulo
bounded real functions supported on a $\nu$-null set) real
non-negative function $d: S_c\to \R_{\geq 0}$ such that
\begin{equation}\label{eq:a}
\chi_t(a)=\sum_{s\in W_0t}d(s)a(s)
\end{equation}
By this formula it is clear that we can extend the application
$\A\ni a\to\chi_t(a)$ for fixed $t\in S_c$ continuously to
an application $C^\infty(cT_u)\ni\phi\to \chi_t(\phi)$
and with this definition we have for all $t\in S_c$ that
\begin{equation}\label{eq:phi}
\chi_t(\phi)=\sum_{s\in W_0t}d(s)\phi(s)
\end{equation}
On the other hand, by Lemma \ref{lem:help} there are for
each $\phi\in C^\infty(cT_u)$ finitely many $a_i\in\A$ and
$\phi_i\in C^\infty(cT_u)^{W_{0,c}}$ such that
$\phi=\sum a_i\phi_i$ and thus (\ref{eq:a}) and (\ref{eq:phi})
yield
\begin{equation}
\chi_t(\phi)=\sum_i \phi_i(t)\chi_t(a_i)
\end{equation}
This shows that for all $\phi\in C^\infty(cT_u)$
the function $t\to\chi_t(\phi)$ is $\nu$-integrable.

There exists a stratification of $cT_u$ by finitely many
locally closed subsets $D_i$ such that for all $x\in D_i$
there exists an open set $U_x\subset D_i$ with the property
that for all $y\in U_x$ we have $W_0y\cap U_x=y$.

For every $x\in D_i\cap S_c$ we can find a
$\phi\in C^\infty(cT_u)$ which is
supported on a small neighborhood $U$ of $x$ such that
$U_x=U\cap D_i$ is as above.
By (\ref{eq:phi}) we see that for such
choice of $\phi$ we have for all $s\in U_x$:
\begin{equation}
d(s)\phi(s)=\chi_s(\phi)
\end{equation}
It follows that the restriction of $d(s)$ to
each $D_i$, and hence $d(s)$, is integrable.

In view of (\ref{eq:start}),(\ref{eq:a}) and (\ref{eq:phi})
this implies that the application
\begin{equation}\label{eq:desym}
\A\ni a\to\Yf^a_c(1)=|W_0/W_{0,c}|\int_{S_c}\sum_{s\in W_0t}
d(s)a(s)\nu(t)
\end{equation}
extends to a positive measure on $cT_u$ with support contained in
$W_0S_c$. By definition these measures have the property (for
all $a\in\A$) that
\begin{equation}
\tau(a)=\sum_{c\in \Cc}\Yf^a_c(1)
\end{equation}
and thus, again by the uniqueness assertion of Lemma \ref{lem:dist},
it follows that
\begin{equation}\label{eq:X=Y}
\Xf_c(a)=\Yf^a_c(1)
\end{equation}
and thus that $\Xf_c$ is indeed a positive measure.
\end{proof}
\begin{lemma}\label{lem:help}
The $C^\infty(cT_u)^{W_{0,c}}$-module $C^\infty(cT_u)$
can be generated by a finite set $E\subset\C[cT_u]$.
\end{lemma}
\begin{proof}
Here we identify $cT_u$ with the compact torus $T_u$ via
$ct\to t$ so that we can speak of the space $\C[cT_u]$ of
Laurent polynomials on $cT_u$. We identify $\A$ with $\C[cT_u]$
via restriction of $a\in\A$ to $cT_u\subset T$.

We have a smooth action of a finite group
$G=W_{0,c}$ on this compact torus which comes from a
linear action of $G$ on the character lattice $X$.
It is then easy to see that there exists a real vector space
$V$ with an orthogonal linear action of $G$ on $V$
and a $G$-equivariant proper embedding $i:cT_u\to V$ such
that $\C[cT_u]=i^*(\C[V])$ and $\C[cT_u]^{G}=i^*(\C[V]^{G})$.
We have $C^\infty(cT_u)=i^*(C^\infty(V))$,
and we can also easily see that (see e.g. \cite[paragraph 5]{M})
$C^\infty(cT_u)^G=i^*(C^\infty(V)^G)$.

Since $G$ is finite, $\C[V]$ is a finitely generated
$\C[V]^{G}$-module. Indeed, we have the standard
$\C[V]^{G}$-module isomorphism
\begin{align*}
\bigoplus_{\pi\in\operatorname{Irr}(G)}
\C[V,V_\pi^*]\otimes V_\pi&\to\C[V]\\
p\otimes v&\to \langle p,v\rangle
\end{align*}
and the $\C[V]^{G}$-modules of equivariant polynomials
$\C[V,V_\pi^*]$ are known to be finitely generated.
Let $F_\pi\subset \C[V,V_\pi^*]$ be a finite set of
generators. By a result of Po\'enaru
\cite[p.106]{P}, it is known that $F_\pi$ also generates
the space of equivariant smooth functions
$C^\infty(V,V_\pi^*)$ as a module over $C^\infty(V)^G$.

Hence similar to the argument above,
$C^\infty(V)$ is generated over $C^\infty(V)^G$ by
the union $E_V\subset\C[V]$ of the sets $E_\pi$ consisting of
the elements
of the form $\langle f,v_i\rangle$ with $f\in F_\pi$ and $v_i$
running over a basis of $V_\pi$.

Via the surjective algebra homomorphism $i^*$ we see
that $E=i^*(E_V)\subset\C[cT_u]$ is a set of generators of
$C^\infty(cT_u)$ over $C^\infty(cT_u)^G$.
\end{proof}
\begin{remark}
In general $\Xf^h_c$ can be of higher order as a
distribution. What goes wrong in the above argument
is the fact that formula \ref{eq:a} does not hold
for the functional $a\to\chi_t(ah)$ for
an arbitrary $h\in\H$. Indeed, the ``off-diagonal'' matrix
elements of $\A$ acting in a unitary irreducible representation
$V$ of $\H$ are not in general linear combinations of characters
of $\A$ if $\A$ acts non-semisimply. This may be the case since
$\A$ is not a $*$-subalgebra. It is of course well known that this
phenomenon indeed occurs.

Using the integral formula for $\tau$ in \ref{sec:trace}
we see in this way that the length of the
indecomposable blocks of the restriction of an irreducible
tempered representation $V$ to $\A$ is bounded by the zero
order of $\Delta$ at the central character of $V$.
\end{remark}
\begin{remark}
In some sense $\Xf$ is a nonsymmetric version
of the Plancherel measure of $\Ze$ with respect to $\tau$.
Its non-symmetry encodes
the distribution of the $\A$-weight spaces in the collection
of tempered representations carried by a given central
character in the support of $\nu$.
\end{remark}
\section{The support of $\nu$}
We now put things together in order to prove
that the support of the Plancherel measure $\nu$ is
equal to $S$. At the same time we derive certain fundamental
properties of the configuration of poles and zeroes
of $\eta$.
\begin{thm}\label{thm:support}
Let $\L$ denote the collection of residual cosets
for $\eta$. The following assertions hold true:
\begin{enumerate}
\item[(A)] For every $\eta$-residual coset $L$ with center $c\in T_v$
there exists a $w\in W_0$ such that $wL^t$ is contained in the support
of the measure $\Xf_{wc}$.
\item[(B)] For all $L\in\L$: $o_L=0$.
\item[(C)] For all $M,L\in\L$ with $L\not=M$: $M^t\not\subset L^t$.
\end{enumerate}
In particular, the support of $\nu$
(and thus the central support of the
Plancherel measure, cf. Corollary \ref{cor:projsup})
equals the union of the
tempered forms $L^t$ of the $\eta$-residual cosets.
\end{thm}
\begin{proof}
The last assertion follows from (A). Indeed, by
(\ref{eq:X=Y}) in the proof of Corollary \ref{cor:Xmeas}
we see that the support of $\Xf_c$ is equal to the support
of $a\to\Yf^a_c(1)$, which by (\ref{eq:desym}) is contained in
the support of $\nu_c=\Yf^1_c$. This is a $W_0$-invariant set,
hence we conclude that the $W_0$-orbit of the support of $\Xf_c$
is contained in the support of $\nu_c$.
Conversely, since
$\nu_c=\Yf^1_c$ is the restriction to $\Ze$ of $\Xf_c$
we see that its support must be contained in
the $W_0$-orbit of the support of $\Xf_c$.
This proves the assertion.

Now we turn to the proof of (A),(B), and (C).
We do this simultaneously with induction to the codimension
of the residual cosets. Consider the following assertions:
\begin{enumerate}
\item[(A-k)] For every $\eta$-residual coset $L$ with center $c\in T_v$
and $\operatorname{codim}(L)\leq k$ there exists a $w\in W_0$ such that
$wL^t$ is contained in the support of the measure $\Xf_{wc}$.
\item[(B-k)] For all $L\in\L$ with $\operatorname{codim}(L)\leq k$: $o_L=0$.
\item[(C-k)] For all $M,L\in\L$ with $M\not=L$ and
$\operatorname{codim}(L)\leq k$: $M^t\not\subset L^t$.
\end{enumerate}
Start of the induction: (B-0) obviously holds.
(C-0) holds too, since $T^t:=T_u$ meets none of the poles
$\a(t)=q_{\a^\vee}$ of $\eta$ (indeed, for this to happen one should have
$q_{\a^\vee}=1$. This is allowed, but in that case the pole will cancel
against a factor of $\Delta(t)$ in de numerator of $\eta$). Finally (A-0)
holds, in view of Proposition \ref{prop:cycle}, and
$\Xf_e(a)$ is simply $\int_{T_u}a\eta$. We note that we can
take $\e^T=1$ here in Proposition \ref{prop:cycle} since
(as we remarked in the proof of (B-0)) no pole meets $T_u$. In
particular, $\Xf_e$ is a smooth, nonzero measure on $T_u$.

Assume(A-l), (B-l) and (C-l) for all $l<k$, and let
$L$ be a residual subspace of codimension $k$.
By $W_0$-invariance it is equivalent to show
(A-k), (B-k) and (C-k) for $L$ or for $wL$ so that
we may and will assume
that $R_L$ is a standard parabolic root subsystem.
We apply Proposition \ref{prop:cycle} to $\tau$.
Depending on the choices of the $\e^M$ this defines
a distribution $\Xf_L$ supported on $L^t$.
As was observed in \cite[Proposition 3.10]{O} the cycle
$\xi_L\subset T_L$ (as in Proposition \ref{prop:cycle}) depends
only on the arrangement
of poles and zeroes of $\eta$ which contain $L$, and we may take
$\xi_L=\xi_{r_L}$, the cycle which is associated to the
residual point $r_L\in T_L$ with respect to the affine Hecke
algebra $\H(X_L,Y_L,R_L,R_L^\vee,F_L)$. In $T_L$ we are thus
working with the rational form $\eta_L$ associated to the based
root datum $\Ri_L=(X_L,Y_L,R_L,R_L^\vee,F_L)$, multiplicity
function $q_L$, and base point $p_L=pT^L_v\cap T_{L,v}\in T_{L,v}$
deep in the negative chamber.

By the induction hypothesis (B-(k-1)) and (C-(k-1))
the collection of residual cosets of $\eta_L$ on $T_L$ satisfies
the conditions of Lemma \ref{lem:nc}. By Lemma \ref{lem:nc} there
exist chambers of the (real projection of the) pole arrangement
of $\eta_L$ such that for a base point $p_L^\prime$ in this chamber
the residue contribution at $r_L$ with respect to $p_L^\prime$ and the
form $\eta_L$ is nonzero and is a distribution of order $o_L$.
By Proposition \ref{prop:wp} there exist $w\in W_L$ such that
$p_L^\prime=w^{-1}(p_L)$ is in one of the
chambers so that the residue at $r_L$ is nonzero.
By the $W_L$ quasi-invariance of
$\eta_L$ this is equivalent to saying that the
residue at $r_{wL}$ is nonzero with respect to the base point
$p_{wL}=pT^{wL}_v\cap T_{wL,v}$ in $T_{wL}$. Hence by Proposition
\ref{prop:cycle} and the text following this Proposition we see that
$\Xf_{wL}$ is nonzero.
By the induction hypothesis $wL^t$ is not contained in
a larger tempered residual coset. By Remark \ref{rem:split}
the (possible) ambiguity of $\Xf_{wL}$ (as a consequence of
choices in the $\e^M$) is limited to a subset
of $wL^t$ of positive codimension. By Corollary \ref{cor:form}(ii)
this implies that $wL^t$ is in the support of $\Xf$, proving (A-k).

By Corollary \ref{cor:Xmeas} we know that $\Xf_{wL}$ is a measure.
Hence Corollary \ref{cor:form2} implies that $o_{wL}=0$.
Since $o_{wL}$ is independent of $w$ (by $W_0$-invariance of
$\eta$) and since $L$ is an arbitrary residual coset with codimension
$k$ this proves (B-k).

Let us finally prove (C-k), arguing by contradiction.
Since $o_{wL}=0$, the text after Proposition \ref{prop:cycle} and
Corollary \ref{cor:form}(ii) imply that $\Xf_{wL}$ is the boundary
value at $wL^t=r_{wL}T^{wL}_u$ of the rational function
$\rho^{wL}(t^{wL})=P^{wL}(r_{wL}t^{wL})/Q^{wL}(r_{wL}t^{wL})$
with
\begin{align*}
P^{wL}(t)&=\Delta^{L}(w^{-1}t)\Delta^{L}(w^{-1}(t^{-1}))\\
Q^{wL}(t)&=D^{L}(w^{-1}t)D^{L}(w^{-1}(t^{-1})).
\end{align*}
By the form of
the factors of $\rho^{wL}$ the intersection of a pole or a zero of
$\rho^{wL}$ is either empty or of real codimension $1$ in $wL^t$.
Assume that there exists a
residual coset $M\not=L$ such that $M^t\subset L^t$.
Since $o_L=0$
this implies that $i_M=o_M+\operatorname{codim}(M)>
o_L+\operatorname{codim}(L)=i_L$.
Hence in this situation there exists a codimension $1$ coset
in $wL^t$ along which $\rho^{wL}$ has a pole of order at least $1$.

Let $wL^{t,reg}\subset wcT_u$ be the locally closed subset which is
the complement in $wL^t$ of the singularities of $\rho^{wL}$ on $wL^t$.
Recall that $\sum_{c\in\Cc}\Xf_{c}$ is a probability Radon measure by
Corollary \ref{cor:Xmeas}. In particular, for all Borel measurable subsets
$U\subset cT_u$ such that $\overline{U}\subset wL^{t,reg}$ we have
\begin{equation}\label{eq:bnd}
0\leq\Xf_{wc}(\chi_U)\leq 1
\end{equation}
On the other hand, outside the locus of singularities of $\rho^{wL}$,
$\Xf_{wc}|_U$ is, up to a nonzero constant $\mu$, equal to the boundary
limit value (coming from some
appropriate wedge which is immaterial at this point) of $\rho^{wL}$
on $U\subset wL^t$, i.e. integration of test functions supported on
$U$ against $\omega^{wL}$. Thus for any open set
$U\subset wL^{reg,t}$ such that $\overline{U}\subset wL^{t,reg}$ we have
\begin{equation}
\Xf_{wc}(\chi_U)=\mu\int_U \rho^{wL}(t)d^{wL}(t)
\end{equation}
Together with (\ref{eq:bnd}) this implies that $\mu \rho^{wL}$ is a
positive integrable function on $wL^{reg,t}$ with respect to
Haar measure. But this contradicts the presence of a pole of $\rho^{wL}$
along a codimension 1 coset in $wL^t$.
Hence our assumption was false, proving (C-k).
This finishes the induction step, and we are done.
\end{proof}
\begin{cor}\label{cor:smooth}
For all $c\in\Cc$, $\Xf_c$ is a sum over the $L\in\L$ such
that $c_L=c$ of the push forward to $cT_u$ of a smooth measure
on $L^t$. In view of Remark \ref{rem:split} the splitting if
$\Xf_c$ as sum of $\Xf_L$ (for $L\in\L$ such that $c_L=c$) such
that each $\Xf_L$ is a smooth measure on $L^t$ is unique. We can
take all $\e^L$ in Proposition \ref{prop:cycle} equal to $1$.
\end{cor}
\section{Implications for root systems}\label{sec:root}
In this section we apply the previous results to prove the
Theorems \ref{thm:1}, \ref{thm:2} and \ref{thm:3} announced in the
introduction.
\begin{thm}\label{thm:mainbdd}
For any coset $L\subset T$ we have $o_L\leq 0$.
Explicitly this means that for all cosets $L\subset T$:
\begin{equation}
|\{\a\in R_0\mid \a|_L=q_{\a^\vee}\}|-|\{\a\in R_0\mid \a|_L=1\}|\leq
\operatorname{codim}(L)
\end{equation}
In particular a coset $L$ is $\eta$-residual in the sense of
Definition \ref{dfn:rescos} if and only if $o_L=0$.
\end{thm}
\begin{proof}
The first assertion is Theorem \ref{thm:support}(B). The second
assertion follows from this. Indeed, let $o_L=0$ and
let $M\supset L$ be the connected component containing $L$
of the intersection of pole hypersurfaces $L_m\subset T$ of $\eta$
which contain $L$.
Then $i_M\geq i_L$ and thus, since $o_L=0$ and $o_M\leq 0$ we have
$\operatorname{codim}(M)\leq\operatorname{codim}(L)
= i_L\leq i_M \leq i_M-o_M=\operatorname{codim}(M)$. It follows that
$L=M$, and that $L$ is residual in the sense of Definition
\ref{dfn:rescos}.
\end{proof}
\begin{proof}[Proof of Theorem \ref{thm:1}]
The statement that a residual point $v$ satisfies $o(v)=0$
is the special case of Theorem \ref{thm:mainbdd}
where $L$ is a point in $T_v$, rephrased additively
(using the exponential isomorphism between $V=\operatorname{Lie}(T_v)$
and the vector group $T_v$, with $k_\a=\log(q_{\a^\vee})$).
Let us prove that there are finitely many residual points.
By Lemma \ref{lem:lines} every residual point $v$
lies on at least one residual line $L$. Since $o_L=0$ we have
for a generic point $w\in L$ that $o(w)=-1$. Hence every residual
line $L$ contains at most finitely many residual points. A residual
line $L$ is uniquely determined by a pair $(R_L,r_L)$ where $r_L$ is
the unique residual point of $R_L$ such that $L=r_L+V^L$.
Hence by induction on the rank the number of
residual lines $L$ is finite. We conclude that the set of residual
points for $R_0$ is finite.
\end{proof}
The following corollary shows that the notion of residual cosets which
we have used here in this note coincides with the original recurrent
definition proposed in \cite{HOH} and \cite{O}:
\begin{cor}
We can alternatively define residual cosets recurrently as follows.
$T$ is a residual coset, and in general a coset $L\subsetneqq T$ is residual
if and only if there exists a residual coset $M\supset L$ such that
$\operatorname{codim}(L)=\operatorname{codim}(M)+1$ and such that
$i_L\geq i_M+1$ (in fact it is clear by Theorem \ref{thm:support}(B)
that only the case $i_L = i_M+1$ will occur).
\end{cor}
\begin{proof} By Theorem \ref{thm:mainbdd}
this is a consequence of \cite[Lemma A.11]{O}.
\end{proof}
\begin{remark}[Theorem \ref{thm:1} and Richardson's Dense Orbit Theorem]
\label{rem:richardson}
In the special case of equal parameters $k_\a=k$
there is an alternative conceptual proof of Theorem \ref{thm:1}.
Without loss of generality we assume that $k_\a=k=2$ for all $\a\in R_0$.
Given $v\in V$ we denote by $R_v\subset R_0$ the root subsystem
$R_v=\{\a\in R_0\mid \a(v)\in 2\mathbb{Z}\}$
and we let $V_v$ be the span of $R^\vee_v$.
Observe that $o(R_0,V,2;v)\leq o(R_v,V_v,2;v)$.
We choose positive
roots $R_{v,+}\subset R_v$ such that $v$ is dominant integral for
$R_v$. Let $G$ be the semisimple complex group of adjoint type
with root system $R_v$, and let $P\subset G$ be the parabolic
subgroup of $G$ whose Lie algebra is
$\mathfrak{p}=V_{v,\mathbb{C}}+\sum_{\a:\a(v)\geq 0}\mathfrak{g}_\a$.
Let $P=LU$ be the Levi decomposition of $P$. By \cite[Proposition
5.8.1]{C} we see that $o(R_v,U_v;v)=
\operatorname{dim}(U/U^\prime)-\operatorname{dim}(L)$.
By Richardson's Dense Orbit Theorem \cite[Theorem 5.2.1]{C}
(also see \cite[Proposition 5.8.2]{C}) applied to $P$ we see
that $o(R_v,U_v;v)\leq 0$, proving
Theorem \ref{thm:1} in this special case.
\end{remark}
\begin{thm}
Let $r\in T$ be a residual point with polar decomposition
$r=cs\in T_vT_u$. Let $r^*=c^{-1}s$. Then
$r^*\in W(R_{0,s})r$, where $R_{0,s}\subset R_0$ is the
root subsystem $\{\a\in R_0\mid \a(s)=1\}$.
\end{thm}
\begin{proof}
The real point $c\in T_v$ is a residual point
for the root datum $(R_{0,s},X,R_{0,s}^\vee,Y)$.
Indeed, it follows straight from the definition
that $o_{\{c\}}(R_{0,s})=o_{\{r\}}(R_0)$. Hence it suffices to
prove the case $s=1$. By Theorem \ref{thm:support} we know
that $W_0r$ is in the support of $\nu$.  Now apply
Proposition \ref{prop:posmeas}.
\end{proof}
\begin{proof}[Proof of Theorem \ref{thm:2}]
Theorem \ref{thm:2} is the special case where $r\in T_v$
is a real residual point in $T$ (rephrased additively).
\end{proof}
\begin{proof}[Proof of Theorem \ref{thm:3}]
Suppose that Theorem \ref{thm:3} does not hold.
Let $L=v_L+V^L\subset V=\operatorname{Lie}(T_v)$
be a residual affine subspace of positive
dimension whose center $v_L$ is a residual point of $V$.
Then $\{\exp(v_L)\}\in T_v$ is a residual point (where we take
$q_{\a^\vee}=\exp(k_\a)$), and $\exp{v_L+iV^L}$ is the tempered
part of a residual coset in $T$ properly containing the residual
point $\exp{v_L}$. This contradicts Theorem \ref{thm:support}(C).
\end{proof}
\begin{remark}[Theorem \ref{thm:3} and the Bala-Carter Theorem]
\label{rem:bala}
Assume we are in the situation of Remark \ref{rem:richardson}.
Let $G$ be the complex semisimple
group of adjoint type with root system $R_0$. By the Bala-Carter
Theorem \cite[Theorem 5.9.5]{C}
the unipotent orbits of $G$ are parameterized by
$G$-conjugacy classes of pairs $(M,P)$ with $M\subset G$ a Levi
subgroup and $P\subset M^\prime$ a distinguished parabolic
subgroup. Using \cite[Proposition B1]{O} it is easy to see that
a center $r_L$ of a residual affine subspace $L=r_L+V^L$ which is
in addition dominant
is the weighted Dynkin diagram of the unipotent orbit $C_{(M,P)}$
with $M=M(R_L)$ and $P=P(r_L)\subset M$ the unique distinguished
parabolic subgroup corresponding to the $R_L$-residual point $r_L$.
By \cite[Proposition 5.6.8]{C} we see that if
$r_L, r_{L_1}$ are dominant and are the respective centers of
two residual affine subspaces $L, L_1$, then $r_L=r_{L_1}$ implies
that $L_1=wL$ for some $w\in W_0$.
\end{remark}
\section{Temperedness of local traces $\chi_t$}
From general principles of harmonic analysis (see \cite{B})
it should be expected that the support of the Plancherel measure $\H$
consists of the set of irreducible tempered representations.
We show here that this
indeed follows from the residue method and the basic technique
of taking the weak constant part of a tempered representation
along a standard parabolic subquotient algebra (see \cite{DO}).
In addition we will find that the central support of the set
of discrete series representations and the central support
of the set of summands of properly induced tempered
representations are disjoint.

The results of this section can be found in \cite{O} but the
proofs in \cite{O} are inaccurate at some points
(based on an earlier inaccuracy that occurred in \cite{HOH0}).
We use the opportunity to correct the arguments here.

First of all we remark that the definition
\cite[Definition 3.9]{O} of the dual central configuration $\L_L$
in $T_L$ is not correct. The correct definition is that $\L_L$
consists of all the cosets $c_LT_{M,v}$ with $M$ in the
intersection lattice generated by the central arrangement $\L^L$
in $T_v$ (and not just the $M\in\L^L$ which are coming from a
residual coset as was written erroneously in
\cite{O}). With this understood
\cite[Proposition 3.12]{O} is correct as stated.
\begin{thm}
For all $t\in S$ the local traces $h\to\chi_t(h)$ are tempered.
\end{thm}
\begin{proof}
The temperedness of the local traces is equivalent, by Casselman's
criterion, to the statement that the real projection to $T_v$ of
the support of the nonsymmetric measure $\Xf$
(which is associated to $\tau|_\A$
via the residue lemma)
is contained in the
closure of the anti-dual cone $T^-_v$ of the positive Weyl chamber.
This is a basic feature of the residue computations. Namely if a
residual coset $wL$ is such that its center $wc_L$ is not in the
closure of $T^-_v$ then $\chi_{wL}=0$ by \cite[Proposition 3.12]{O}.
\end{proof}
The next result gives an alternative for the proof of
\cite[Lemma 3.31]{O}.
We remark that the proof as in \cite{O}, directly using
\cite[Proposition 3.12]{O},
can be corrected too, but this becomes somewhat more complicated
by the weakening of \cite[Proposition 3.12]{O} mentioned above.
\begin{cor}
Let $S^d$ denote the union of the $W_0$-orbits of residual
points, and $S^c$ the union of the $W_0$-orbits of the tempered forms
of the residual cosets of positive dimension.
Then $S = S^d \sqcup S^c$, $S^d$ is the central support of the set
of discrete series representations of $\H$ and $S^c$ is the central
support of the set of irreducible tempered characters
which are not in the discrete series.
\end{cor}
\begin{proof}
We have $S = S^d \sqcup S^c$ by Theorem \ref{thm:support}(C).
Suppose by induction on the rank of $X$ that the statement is true for
all proper standard parabolic subquotient algebras
$\H(X_P,Y_P,R_P,R_P^\vee)$.
By the theory of the weak constant term (see \cite{DO})
it follows that if an irreducible tempered representation
$\pi$ is not a discrete series representation then $\pi$
is a summand of a tempered standard module which is induced
from a discrete series representation of a proper standard
parabolic subquotient algebra $\H(X_P,Y_P,R_P,R_P^\vee)$.
By the induction hypothesis and the fact that
$S^d\cap S^c=\emptyset$ the central character of $\pi$
will belong to $S^c$.
Thus the tempered irreducible characters supported by $S^d$
are discrete series characters.

It is known that an irreducible representation $\pi$ of $\H$
is a discrete series iff the mass of $\pi$
in the Plancherel measure is positive (see \cite[Theorem 2.25]{O}).
Hence by the
absolute continuity of the measure $\nu$ (Corollary \ref{cor:smooth})
and the fact that $\nu$ is the push-forward of the Plancherel measure
(see Corollary \ref{cor:projsup}) to $\operatorname{MaxSpec}(\Ze)$
(equipped with the analytic topology) there can be no discrete series
supported by $S^c$.
Therefore the above and Theorem \ref{thm:support} imply that
the central support of the discrete series representations
is equal to $S^d$.
\end{proof}

\end{document}